\batchmode
\makeatletter
\def\input@path{{"/Users/russw/Documents/Research/mypapers/Antichain cutsets in real-ranked lattices/"}}
\makeatother
\documentclass[12pt,oneside,english,lowtilde]{amsart}
\usepackage[T1]{fontenc}
\usepackage[latin9]{inputenc}
\usepackage{babel}
\usepackage{url}
\usepackage{enumitem}
\usepackage{amstext}
\usepackage{amsthm}
\usepackage{amssymb}
\usepackage{graphicx}
\usepackage{geometry}
\usepackage[dvips,pdfusetitle,
 bookmarks=true,bookmarksnumbered=false,bookmarksopen=false,
 breaklinks=false,pdfborder={0 0 1},backref=false,colorlinks=false]
 {hyperref}
\usepackage{breakurl}

\makeatletter
\numberwithin{equation}{section}
\numberwithin{figure}{section}
\theoremstyle{plain}
\newtheorem{thm}{\protect\theoremname}[section]
\theoremstyle{plain}
\newtheorem{question}[thm]{\protect\questionname}
\theoremstyle{plain}
\newtheorem{cor}[thm]{\protect\corollaryname}
\theoremstyle{definition}
\newtheorem{example}[thm]{\protect\examplename}
\theoremstyle{plain}
\newtheorem{lem}[thm]{\protect\lemmaname}
\theoremstyle{remark}
\newtheorem{rem}[thm]{\protect\remarkname}

\usepackage{lmodern}

\makeatother

\providecommand{\corollaryname}{Corollary}
\providecommand{\examplename}{Example}
\providecommand{\lemmaname}{Lemma}
\providecommand{\questionname}{Question}
\providecommand{\remarkname}{Remark}
\providecommand{\theoremname}{Theorem}

\begin{document}
\global\long\def\cc{\mathbb{C}}%

\global\long\def\rr{\mathbb{R}}%

\global\long\def\qq{\mathbb{Q}}%

\global\long\def\ff{\mathbb{F}}%

\global\long\def\link{\operatorname{link}}%

\title{Antichain cutsets in real-ranked lattices}
\author{Stephan Foldes and Russ Woodroofe}
\thanks{Work of the second author is supported in part by the Slovenian Research
Agency (research program P1-0285 and research projects J1-9108, N1-0160,
J1-2451).}
\address{Miskolci Egyetem, 3515 Miskolc-Egyetemvaros, Hungary}
\email{foldes.istvan@uni-miskolc.hu}
\address{Univerza na Primorskem, Glagoljaška 8, 6000 Koper, Slovenia}
\email{russ.woodroofe@famnit.upr.si}
\urladdr{\url{https://osebje.famnit.upr.si/~russ.woodroofe/}}
\begin{abstract}
We show that in a rank supersolvable lattice that is graded by a bounded
real interval, any antichain cutset is a level set for some appropriately
constructed grading. As a consequence, given an antichain cutset in
any of the measurable Boolean lattice, a continuous partition lattice,
or a continuous projective geometry, we may find a grading in which
the cutset is a level set.
\end{abstract}

\maketitle

\section{\label{sec:Introduction}Introduction}

The Boolean lattice on a finite set is a basic example in lattice
theory, with almost any nice property that one could ask for. The
Boolean lattice on an infinite set, on the other hand, is much less
well-behaved. For example, it is difficult or impossible to extend
gradedness to this situation; in particular, ordinal theory gives
that there are many isomorphism classes of maximal chains. Better-behaved
is the \emph{measurable Boolean lattice} on a complete measure space
$(X,\mu)$ with a finite measure. In this lattice, we restrict ourselves
to the measurable subsets of $X$, and consider two measurable subsets
to be equivalent if their symmetric difference has measure zero. If
the measure does not have any points of positive measure (or, in the
other extreme, if $\mu$ is the counting measure), then $\mu$ gives
an analogue of a rank function.

The situation with the partition lattice on an infinite set is similar:
the full lattice of partitions on $[0,1]$ is too large to admit a
nice rank function, and the same holds even for the lattice of partitions
into measurable parts. But Von Neumann in \cite{VonNeumann:1936a,VonNeumann:1936b}
developed machinery to build continuous extensions of lattices of
subspaces and similar families of modular lattices. Indeed, the measurable
Boolean lattice on the interval $[0,1]$ can be viewed as an example
of Von Neumann's construction. Björner in \cite{Bjorner:1987} extended
the machinery of Von Neumann to geometric and semimodular lattices.
A main application of Björner's extension was a continuous extension
of finite partition lattices. Haiman in \cite{Haiman:1994} gave a
more concrete construction of a continuous partition lattice, built
from measurable sets in a similar manner to the measurable Boolean
lattice. Björner's lattice embeds in that of Haiman.

The current paper came out of an effort to extend some of our earlier
work on antichain cutsets for discrete posets \cite{Foldes/Woodroofe:2013}
to the continuous situation, covering (at least) all of the examples
we have so far discussed. An \emph{antichain cutset} in a lattice
$L$ is an antichain $A$ so that every maximal chain nontrivially
intersects $A$. Antichain cutsets were first studied by Rival and
Zaguia \cite{Rival/Zaguia:1985}. The structure of antichain cutsets
is much more prescribed than that of antichains more generally. For
example, the family of antichain cutsets of a graded, finite, and
strongly connected poset coincides exactly with that of the level
sets of the poset \cite[Theorem 2]{Foldes/Woodroofe:2013}. We contrast
with the situation for general antichains, where e.g. a finite partition
lattice typically has antichains that are larger than any level set
\cite{Graham:1978}.

Unlike the finite situation, where a rank function may be recovered
in a straightforward way from the poset structure, the continuous
lattices that we are interested in may admit gradings that are quite
different from one another. For example, let $L$ be the measurable
Boolean lattice on $[0,1]$ with Lebesgue measure $\mu$. Then $\mu$
is a rank function (see Section~\ref{sec:Background} for a definition),
and given a measurable positive function $f$, the function $\nu(U)=\int_{U}f\,d\mu$
yields another rank function. The level sets of these two rank functions
are generally quite different from each other. Since a level set of
any rank function is necessarily an antichain cutset, it is too much
to hope that every antichain cutset is a level set in some chosen
rank function.

We believe that the following question is reasonable to ask, and of
some interest.
\begin{question}
\label{que:ACClevel}For which classes of graded posets does it hold
that, for each antichain cutset $A$, there is a grading $\sigma$
(possibly distinct from an initially presented one) so that $A$ is
a level set under $\sigma$?
\end{question}

The first author showed in \cite{Foldes:2024} that the condition
asked for in Question~\ref{que:ACClevel} holds for space-time causality
orders.

In the current paper, we will show that the condition of this question
holds for a reasonably general class of lattices, defined as follows.
Let $L$ be a lattice with an understood rank function $\rho:L\to R$,
for $R$ some subset of $\rr$, or more generally of the extended
reals $[-\infty,\infty]$. We say that an element $m$ of $L$ is
\emph{rank modular} if for every element $x$ of the lattice, it holds
that $\rho(x\vee m)+\rho(x\wedge m)=\rho(x)+\rho(m)$. Equivalently,
the same relation holds whenever $x$ is incomparable with $m$. We
say that $L$ is \emph{rank supersolvable} if it has a maximal chain
consisting of rank modular elements (a \emph{chief chain}).

In the case of a finite lattice with the natural rank function, rank
supersolvability is one of several equivalent conditions for supersolvability
\cite{Foldes/Woodroofe:2021}. We do not know if these conditions
are equivalent in the non-discrete situation; see also Lemma~\ref{lem:ElmSS}
and the subsequent discussion.

Our main result is as follows.
\begin{thm}
\label{thm:Main}Let $R\subseteq\rr$ be a bounded interval. If $L$
is a rank supersolvable lattice with respect to the $R$-grading $\rho$,
and $A$ is an antichain cutset for $L$, then there is some grading
$\sigma$ under which $A$ is a level set.
\end{thm}

Under an additional continuity condition, we can extend Theorem~\ref{thm:Main}
to unbounded intervals. A precise statement of this extension may
be found in Theorem~\ref{thm:MainExt} below.
\begin{cor}
For any antichain cutset $A$ in
\begin{enumerate}
\item the measurable Boolean lattice over $[0,1]$ with the Lesbesgue measure,
or
\item the continuous projective geometries of Von Neumann, or
\item the continuous partition lattice of either Björner or Haiman,
\end{enumerate}
there is some grading $\sigma$ under which $A$ is a level set.
\end{cor}

The chief chain of $L$ (with respect to $\rho$) does not necessarily
satisfy the rank modular condition with respect to the grading $\sigma$
in our construction; thus, it is not necessarily rank supersolvable
in the grading $\sigma$. We do not know whether, in the situation
of Theorem~\ref{thm:Main}, it is possible to construct a grading
$\sigma$ where both $A$ is a level set and also rank modularity
of the chief chain is preserved.

In contrast, the (full) Boolean lattice on a infinite set admits only
the trivial antichain cutsets $\{\hat{0}\}$ and $\{\hat{1}\}$. For
a countably infinite set, this is straightforward, as follows. The
subposet of all finite subsets forms a (non-complete) sublattice.
By countability, adjoining $\hat{1}$ to a maximal chain in this sublattice
yields a chain that is maximal in the full Boolean lattice. The subposet
of all subsets with a finite complement is dual, and entirely similar.
It now follows by the main result of \cite{Foldes/Woodroofe:2013}
that any putative nontrivial antichain cutset must contain both all
sets of some size $n$, and likewise also all sets whose complement
has some size $n'$. But this clearly contradicts the antichain property!
In a recent preprint \cite{Ginsburg/Sands:2025UNP}, Ginsburg and
Sands have shown the stronger result that any cutset in the Boolean
lattice over a set of infinite cardinality $\kappa$ must contain
a chain having strictly larger cardinality than $\kappa$.

A recent paper of Engel, Mitsis, Pelekis and Reiher \cite{Engel/Mitsis/Pelekis/Reiher:2020}
considers a somewhat different extension of discrete results on antichains
to a continuous and/or analytic setting. \medskip{}

This paper is organized as follows. In Section~\ref{sec:Background},
we define $R$-gradings, and review the details of the construction
of Von Neumann and Björner. In Section~\ref{sec:RMandSS}, we extend
facts about rank modularity and supersolvability from the discrete
case to the current setting. In Section~\ref{sec:ProofMainThm},
we prove (an extension of) Theorem~\ref{thm:Main}.

\section*{Acknowledgements}

We'd like to thank Andrej Bauer, Sándor Radeleczki, and Alex Simpson
for helpful comments.

\section{\label{sec:Background}Definitions and background}

Let $R$ be a subset of $\rr$, or more generally of the extended
reals $[-\infty,\infty]$. An \emph{$R$-rank function} or \emph{$R$-grading}
for the poset $L$ is a function $L\to R$ that restricts to an order
isomorphism on every maximal chain of $L$. Thus, a finite poset is
graded (in the usual sense) if and only if it has an $[n]$-grading.
\begin{example}
Given a complete measure space $(X,\mu)$, we take the measurable
Boolean lattice $\mathcal{B}(X,\mu)$ to be the set of measurable
sets up to the identification of pairs of sets whose symmetric difference
has measure zero. If $\mu$ is a finite measure with no points of
positive measure, then $\mu:\mathcal{B}(X,\mu)\to[0,\mu(X)]$ is a
$[0,\mu(X)]$-grading. We notice also, on the other hand, that the
rank function on a finite Boolean lattice is exactly the counting
measure on the underlying set.
\end{example}

The measurable Boolean lattice on the interval $[0,1]$ with Lebesgue
measure is a natural continuous analogue of the Boolean lattice on
$n$ points.

Our other examples come from a construction originally due to Von
Neumann for modular lattices \cite{VonNeumann:1936a,VonNeumann:1936b},
and extended by Björner to semimodular lattices in \cite{Bjorner:1987},
see also \cite{Bjorner/Lovasz:1987}. We review this construction.
There are two basic ingredients: a direct limit, and a metric completion.

For the direct limit, we start with a graded lattice $L_{n}$ of rank
$n$ for each positive integer $n$. It will be convenient to require
every $L_{n}$ to be semimodular. We renormalize the standard rank
function on $L_{n}$ by dividing by $n$: thus, the top element $\hat{1}$
has renormalized rank $1$, an element in a middle rank has renormalized
rank of roughly $1/2$, etc.

Whenever $k$ properly divides $n$, we consider a lattice embedding
$\varphi^{k,n}:L_{k}\to L_{n}$ that respects the renormalized rank.
We require that, when $k$ divides $m$ and $m$ divides $n$, it
holds that $\varphi^{k,n}$ agrees with the composition of $\varphi^{k,m}$
and $\varphi^{m,n}.$ Now we take the directed limit $L_{(\infty)}$
of this system of lattices and embeddings over the positive integers
with the divisibility order. In other words, we take a union over
all $L_{n}$, and then identify earlier $L_{k}$'s with sublattices
in later $L_{n}$'s. The renormalized rank function carries through
to give a $([0,1]\cap\qq)$-grading on $L_{(\infty)}$.

We now take a metric completion to get a $[0,1]$-graded lattice.
We need to have a metric for this purpose. It has long been known
that in a semimodular lattice with rank function $r$, the function
$d(x,y)=2r(x\vee y)-r(x)-r(y)$ is a metric \cite{Monjardet:1981}.
Indeed, we may view this as the length of the ``up-down path'' from
$x$ up to $x\vee y$, and from $x\vee y$ back down to $y$. If every
$L_{n}$ is semimodular, then these metrics pass to a metric on $L_{(\infty)}$.
Now $L_{\infty}$ is the metric completion of $L_{(\infty)}$ under
this metric. We recover a rank function by letting $\rho(x)=d(\hat{0},x)$;
by definition of the distance function, this $\rho$ extends the renormalized
rank function on $L_{(\infty)}$.
\begin{example}
Von Neumann defines a continuous projective geometry over a field
$\ff$ \cite[Section 5]{VonNeumann:1936a}; see also \cite[Chapter X.5]{Birkhoff:1967}.
In the language above, this may be described as follows. Let $L_{n}$
be the lattice of all subspaces of the vector space $\ff^{n}$. When
$n$ is a multiple of $k$, view $\ff^{n}$ as $\ff^{k}\oplus\ff^{k}\oplus\cdots\oplus\ff^{k}$.
Now define $\varphi$ by mapping a subspace $W$ of $\ff^{k}$ to
$W\oplus\cdots\oplus W$. It is easy to see that $\varphi$ preserves
renormalized rank, and is a lattice mapping. Every element of the
metric completion is rank modular with respect to $\rho$.
\end{example}

\begin{example}
\label{exa:MeasurableBooleanConst}The measurable Boolean lattice
on $[0,1]$ may be built from this machinery, as follows \cite[Section 8]{VonNeumann:1936b}.
Let $L_{n}$ be the Boolean lattice of all subsets of $[n]$. When
$n$ is a multiple of $k$, view $[n]$ as $[k]\times[n/k]$. Define
$\varphi$ by mapping $S\subseteq[k]$ to $S\times[n/k]$, and take
the directed limit and metric completion.

The identification with subsets of $[0,1]$ comes by identifying $i\in[n]$
with the interval $(\frac{i-1}{n},\frac{i}{n}]$. Since the Lebesgue
measure of a union of $m$ intervals of length $\frac{1}{n}$ is $\frac{m}{n}$,
the renormalized rank of $S\subseteq[n]$ agrees with the Lebesgue
measure. Now $L_{(\infty)}$ consists of all sets that may be written
as unions of half-open intervals with rational endpoints. The metric
completion gives us unions of such sets, recovering (up to measure
zero) all Borel sets.
\end{example}

\begin{example}
Björner \cite{Bjorner:1987} builds a measurable partition lattice
using a similar approach. He identifies elements of an underlying
$(n+1)$-element set with the set of intervals 
\[
\left\{ \{0\},(0,\frac{1}{n}],(\frac{1}{n},\frac{2}{n},\dots,(\frac{n-1}{n},1]\right\} .
\]
Then when $n$ is a multiple of $k$, we identify each interval of
length $\frac{1}{k}$ with $\frac{n}{k}$ intervals of length $\frac{1}{n}$
in a similar manner as in Example~\ref{exa:MeasurableBooleanConst}.
Rank supersolvability of every finite partition lattice yields rank
supersolvability of the limit object \cite[Theorem 2]{Bjorner:1987}.
\end{example}

We remark that while one might call $[0,1]$-graded or $\rr$-graded
posets ``continuously graded'', they do not appear to have any relationship
to speak of with the continuous posets of \cite{Gierz/Hofmann/Keimel/Lawson/Mislove/Scott:2003}
that are studied in computability theory. Indeed, while the measurable
Boolean lattice on $[0,1]$ is a main object for our purposes here,
it already fails to have interesting way-below relationships. When
every element of the lattice $L$ is rank modular, the $R$-grading
gives $L$ the structure of a metric lattice in the sense of \cite{Birkhoff:1967}.

Haiman gives explicit descriptions of Björner's measurable partition
lattice as well as of related lattices \cite{Haiman:1994}. Haiman's
lattices are again rank supersolvable \cite[Lemma 2.15]{Haiman:1994}.
A recent preprint of Contreras Mantilla and Sinclair \cite{Contreras-Mantilla/Sinclair:2025UNP}
discusses the Von Neumann and Björner construction from a model theory
perspective. See also the paper \cite{Vershik/Yakubovich:2001} of
Vershik and Yakubovich, and the survey article of Björner \cite{Bjorner:2019}.

\section{\label{sec:RMandSS}Rank modularity and rank supersolvable lattices}

\subsection{Continuity}

We begin with some basic results about rank supersolvable lattices,
extending the well-known theory of finite supersolvable lattices \cite{Foldes/Woodroofe:2021,McNamara:2003,Stanley:1972,Stern:1999}.

Throughout this section, let $R$ be a subset of $\rr$, or more generally
of the extended reals $[-\infty,\infty]$. Let $L$ be an $R$-graded
lattice. Let $\mathbf{m}=\left\{ m_{\lambda}:\lambda\in R\right\} $
be a maximal chain in $L$ with $\rho(m_{\lambda})=\lambda$, which
we will take to be the chief chain of $L$ when it is rank supersolvable.
We place the standard topology of $R\subseteq[-\infty,\infty]$ on
each maximal chain.

We start with an easy but useful consequence of rank modularity, not
requiring rank supersolvability.
\begin{lem}
\label{lem:RMbalance}If $w<z$ and $m'<m$ are elements of $L$,
where $m$ and $m'$ are rank modular, then
\begin{align*}
\rho(m\wedge z)+\rho(m\vee z)-\left(\rho(m\wedge w)+\rho(m\vee w)\right) & =\rho(z)-\rho(w),\text{ and}\\
\rho(m\wedge z)+\rho(m\vee z)-\left(\rho(m'\wedge z)+\rho(m'\vee z)\right) & =\rho(m)-\rho(m').
\end{align*}
\end{lem}

\begin{proof}
Compute! We have $\rho(m\wedge z)+\rho(m\vee z)=\rho(m)+\rho(z)$,
and similarly $\rho(m\wedge w)+\rho(m\vee w)=\rho(m)+\rho(w)$. The
$\rho(m)$ terms cancel, giving the first equality. The second equality
is entirely similar.
\end{proof}
As in the finite case, the difference in ranks $\rho(z)-\rho(w)$
may be viewed as giving the height of the interval $[w,z]$. We recover
from Lemma~\ref{lem:RMbalance} several bounds on such differences.
\begin{cor}
\label{cor:RMdiamondbounds}If $w<z$ and $m'<m$ are elements of
$L$, where $m$ and $m'$ are rank modular, then the following inequalities
hold
\begin{align*}
\rho(m\wedge z)-\rho(m\wedge w)\leq\rho(z)-\rho(w),\qquad & \rho(m\vee z)-\rho(m\vee w)\leq\rho(z)-\rho(w)\\
\rho(m\wedge z)-\rho(m'\wedge z)\leq\rho(m)-\rho(m'),\qquad & \rho(m\vee z)-\rho(m'\vee z)\leq\rho(m)-\rho(m').
\end{align*}
\end{cor}

\begin{proof}
The left-hand sides of the inequalities are all nonnegative by order
considerations. The left-hand sides of the first row sum to the right
hand side(s) of the first row; similarly for the second row.
\end{proof}
We immediately recover two continuity properties.
\begin{cor}
\label{cor:JoinMeetCtntyC}Let $\mathbf{c}=\left\{ c_{\kappa}:\kappa\in R\right\} $
be a maximal chain in $L$, with $\rho(c_{\kappa})=\kappa$. If $m$
is a rank modular element of $L$, then the following maps are Lipschitz
(hence continuous at every point in $R\cap\rr$): 
\[
\begin{array}{l}
R\to R\\
\kappa\mapsto\rho(m\wedge c_{\kappa})
\end{array},\text{\ensuremath{\quad}and\ensuremath{\quad}\ensuremath{\begin{array}{l}
R\to R\\
\kappa\mapsto\rho(m\vee c_{\kappa})
\end{array}}}.
\]
\end{cor}

\begin{cor}
\label{cor:JoinMeetCtntyM}If $L$ is rank supersolvable with chief
chain $\mathbf{m}$, and $z$ is an element of $L$, then the following
maps are Lipschitz (hence continuous at every point in $R\cap\rr$):
\[
\begin{array}{l}
R\to R\\
\lambda\mapsto\rho(m_{\lambda}\wedge z)
\end{array}\text{,\ensuremath{\quad}and\ensuremath{\quad}\ensuremath{\begin{array}{l}
R\to R\\
\lambda\mapsto\rho(m_{\lambda}\vee z)
\end{array}}}.
\]
\end{cor}

An alternate statement of Corollary~\ref{cor:JoinMeetCtntyC} is
that the map formed by composing the three maps $\rho\vert_{\mathbf{c}}^{-1}$,
meet with $m$, and $\rho$ is Lipschitz (hence continuous) as a map
$R\to R$; similarly with Corollary~\ref{cor:JoinMeetCtntyM}.
\begin{rem}
In the case where $L$ is semimodular, so that $d(x,y)=2\rho(x\vee y)-\rho(x)-\rho(y)$
is a metric, then it is a general fact that the join operation is
Lipschitz in this metric (or a product of this metric with itself)
over the entire lattice \cite{Bjorner:1987}. The meet operation,
on the other hand, is not continuous even in the partition lattice.
We can extend Corollary~\ref{cor:JoinMeetCtntyC} to show that meet
with a rank modular $m$ in a semimodular lattice is Lipschitz in
the metric, since 
\begin{align*}
d(m\wedge x,m\wedge y) & =2\rho((m\wedge x)\vee(m\wedge y))-\rho(m\wedge x)-\rho(m\wedge y)\\
 & \leq2\rho(m\wedge(x\vee y))-\rho(m\wedge x)-\rho(m\wedge y)\\
 & =2\rho(x\vee y)-\rho(x)-\rho(y)-2\rho(m\vee x\vee y)\\
 & \leq d(x,y).
\end{align*}
It is likely that there are more general continuity statements, but
since these are incidental to our main purpose, we do not develop
this idea further.
\end{rem}

Although infinite lattices may not be bounded, this is easy to control.
The following is immediate from definition.
\begin{lem}
The lattice $L$ has a (unique) maximum element $\hat{1}$ if and
only if $R$ has a maximum element. Similarly for minimum elements.
\end{lem}

Notice that if $R=[-\infty,\infty]$ and $\hat{1}$ is a join irreducible
or $\hat{0}$ is a meet irreducible, then the rank modular relation
involves an indeterminate form only when we compare the minimum $\hat{0}$
and maximum $\hat{1}$ elements, which causes no confusion.
\begin{lem}
\label{lem:SSBounded}If $L$ (and $R$) has no maximum element, then
we can extend $L$ to a $(R\mathop{\cup}\sup R)$-graded lattice by
adding a maximum element $\hat{1}$ to $L$ and taking $\rho(\hat{1})=\sup R$.
If $L$ is rank supersolvable with chief chain $\mathbf{m}$, then
$L\cup\{\hat{1}\}$ is rank supersolvable with chief chain $\mathbf{m}\cup\{\hat{1}\}$.

Similarly for minimum elements.
\end{lem}

\begin{proof}
It is immediate that the extension of $\rho$ is a grading. Since
$\hat{1}\vee m=\hat{1}$ and $\hat{1}\wedge m=m$, the element $\hat{1}$
is rank modular.
\end{proof}
Corollaries~\ref{cor:JoinMeetCtntyC} and \ref{cor:JoinMeetCtntyM}
give continuity of certain maps at every real point of $R$. If $R$
contains $\infty$ or $-\infty$, then additional care is required.
Recall from a first topology course that, since the extended reals
are homeomorphic to $[-1,1]$, a map $\varphi:[-\infty,\infty]\to[-\infty,\infty]$
is continuous if and only if it respects convergence of sequences.
As an immediate consequence, we obtain the following.
\begin{lem}
\label{lem:CtsAtInfty}An increasing map $\varphi:[-\infty,\infty]\to[-\infty,\infty]$
is continuous at $\infty$ (or \linebreak{}
respectively at $-\infty$) if and only if $\sup\left\{ \varphi(\alpha):\alpha<\infty\right\} =\varphi(\infty)$
(or respectively, if\linebreak{}
 $\inf\left\{ \varphi(\alpha):\alpha>-\infty\right\} =\varphi(-\infty)$).
\end{lem}

We remark that there are supersolvable (indeed, modular) lattices
where the maps of Corollaries~\ref{cor:JoinMeetCtntyC} and \ref{cor:JoinMeetCtntyM}
fail to be continuous at $\pm\infty$.
\begin{example}
\label{exa:DisCtsAtInfty} Consider the product lattice $\rr\times\rr$
with $\rr$-grading $\rho(a,b)=a+b$, and let $L=\rr\times\rr\cup\{\hat{0},\hat{1}\}$.
Then the join (meet) is obtained by taking entry-wise maximums (minimums),
and every element of $L$ is rank modular. Now $\mathbf{c}=\left\{ (0,b):b\in\rr\right\} \cup\{\hat{0},\hat{1}\}$
is a maximal chain where $\sup\left\{ (0,b)\wedge(1,0):b\in\rr\right\} =(0,0)$,
while $\hat{1}\wedge(1,0)=(1,0)$. In particular, the map $b\mapsto\rho((0,b)\wedge(1,0))$
fails to be continuous at $\infty$. An entirely similar argument
shows that the map $b\mapsto\rho((0,b)\vee(-1,0))$ fails to be continuous
at $-\infty$.
\end{example}

\begin{example}
\label{exa:BddMeasBool}Consider the lattice of all bounded Lebesgue-measurable
sets in $\rr$ (up to measure zero), with Lebesgue measure as a $[0,\infty)$-grading.
By modularity of measures, every element is rank modular. Then the
intervals $m_{\lambda}=[-\lambda,\lambda]$ form a chief chain. Since
every bounded measurable set is contained some $m_{\lambda}$, we
have that $\sup_{\lambda}\left\{ m_{\lambda}\wedge y\right\} =y$.
Thus, the first map of Corollary~\ref{cor:JoinMeetCtntyM} is continuous.

On the other hand, the intervals $c_{\kappa}=[1,1+\kappa]$ also form
a maximal chain, and $\sup_{\kappa}\left\{ c_{\kappa}\wedge[-1,1]\right\} =\{1\}\subsetneq[-1,1]$.
Thus, the first map of Corollary~\ref{cor:JoinMeetCtntyC} fails
to be continuous at $\infty$.
\end{example}

\subsection{Real-graded vs discrete-graded supersolvable lattices}

For a finite lattice, there are several conditions that are equivalent
to rank supersolvability. All are based on chains of elements satisfying
some modularity condition. Rank functions in the non-discrete case
are less rigid than those in the finite case, and that an element
that is rank modular with respect to one rank function does not generally
imply that it will be rank modular with respect to another.
\begin{example}
If $\rho(x)$ is a $[0,1]$-grading on $L$, then $\left(\rho(x)\right)^{2}$
is another $[0,1]$-grading. If $m$ is rank modular with respect
to $\rho$, however, then (by the failure of the Freshman's Dream
identity in characteristic $0$) it typically will not be rank modular
with respect to $\rho^{2}$.
\end{example}

\begin{example}
If $f:\rr\to\rr$ is a strictly increasing function, and $\rho$ is
an $R$-grading, then $f\circ\rho$ is an $f(R)$-grading. If $f$
is affine linear, then $m$ is rank modular with respect to $\rho$
if and only if it is rank modular with respect to $f\circ\rho$. If
$f$ is not affine linear, then it typically does not preserve rank
modularity of elements.
\end{example}

Going from rank modularity to other modularity conditions is easier.
The next lemma gives that a rank modular element is ``left modular''.
\begin{lem}
\label{lem:RMgivesLM}If $m$ is rank modular in $L$, then for any
$w<z$ in $L$, it holds that 
\[
(w\vee m)\wedge z=w\vee(m\wedge z).
\]
\end{lem}

\begin{proof}
Apply exactly the same argument as in \cite[Lemma 4.1]{Foldes/Woodroofe:2021}.
\end{proof}
Thus, if $m$ is rank modular, then we may write $w\vee m\wedge z$
without parentheses.

A maximal chain consisting of rank modular elements gives somewhat
stronger lattice conditions:
\begin{lem}
\label{lem:ElmSS}Let $L$ be rank supersolvable with chief chain
$\mathbf{m}$. For $z<w\in L$ and $m_{\lambda}<m_{\kappa}\in\mathbf{m}$,
it holds that 
\begin{align*}
(z\vee m_{\lambda})\wedge w & =z\vee(m_{\lambda}\wedge w),\text{ and}\\
(m_{\lambda}\vee z)\wedge m_{\kappa} & =m_{\lambda}\vee(z\wedge m_{\kappa}).
\end{align*}
\end{lem}

\begin{proof}
The first identity is a special case of Lemma~\ref{lem:RMgivesLM}.

For the second identity, we notice that since $(m_{\lambda}\vee z)\wedge m_{\kappa}\geq m_{\lambda}\vee(z\wedge m_{\kappa}),$
it suffices to show that the ranks agree. Now, since $m_{\lambda}\vee z\vee m_{\kappa}=z\vee m_{\kappa}$
and $m_{\lambda}\wedge z\wedge m_{\kappa}=m_{\lambda}\wedge z$, we
get that 
\begin{align*}
\rho((m_{\lambda}\vee z)\wedge m_{\kappa}) & =\rho(m_{\lambda}\vee z)+\kappa-\rho(z\vee m_{\kappa})\\
\rho(m_{\lambda}\vee(z\wedge m_{\kappa})) & =\lambda+\rho(m_{\kappa}\wedge z)-\rho(m_{\lambda}\wedge z).
\end{align*}
Now Lemma~\ref{lem:RMbalance} gives the difference of these ranks
to be zero, as desired.
\end{proof}
We observe in passing that since the sublattice generated by a finite
modular chain and another finite chain forms a distributive lattice
\cite{Foldes/Woodroofe:2021,Stanley:1972}, the sublattice of $L$
generated by (finite) combinations of $\mathbf{m}$ and another chain
$\mathbf{c}$ is distributive.

In the finite case, the converse of Lemma~\ref{lem:ElmSS} also holds
true \cite[Theorems 1.2 and 1.4]{Foldes/Woodroofe:2021}. The difficulty
in extending this converse to the non-discrete case is in constructing
a nice enough rank function.
\begin{question}
Given an $R$-graded lattice $L$ and a chain $\mathbf{m}$ satisfying
the conclusion of Lemma~\ref{lem:ElmSS}, is it possible to find
a rank function under which all elements of $\mathbf{m}$ are rank
modular?
\end{question}

We close this section by extending one more result from finite lattices
to $R$-graded lattices: rank modular elements project into intervals
nicely.
\begin{lem}
If $m$ is a rank modular element of $L$, then $w\vee m\wedge z$
is a rank modular element of the interval $[w,z]$ for any $w<z$
in $L$.
\end{lem}

\begin{proof}
It suffices by duality and iteration to show that $m\wedge z$ is
rank modular in the interval $[\hat{0},z]$. For $w<z$, we desire
to have $\rho(w\vee m\wedge z)=\rho(w)+\rho(m\wedge z)-\rho(m\wedge w)$.
But the rank modular identity with $m$ and $w\vee m\wedge z$ gives
us that 
\[
\rho(m\vee(w\vee m\wedge z))+\rho(m\wedge(w\vee m\wedge z))-\rho(m)=\rho(w\vee m\wedge z).
\]
 Noticing that $m\vee(w\vee m\wedge z)=m\vee w$ and $m\wedge(w\vee m\wedge z)=m\wedge z$,
and expanding $\rho(m\vee w)$ as $\rho(m)+\rho(w)-\rho(w\wedge m)$
now yields the desired.
\end{proof}

\section{\label{sec:ProofMainThm}Proof of Theorem~\ref{thm:Main}}

Throughout this section, let $L$ be a rank supersolvable lattice
with a given $R$-grading $\rho$ over some interval $R\subseteq\rr$.
We may assume without loss of generality by Lemma~\ref{lem:SSBounded}
that $R$ is compact (passing to the extended reals if necessary),
and so that $L$ has a minimum element $\hat{0}$ and maximum element
$\hat{1}$. Let $\mathbf{m}=\left\{ m_{\lambda}:\lambda\in R\right\} $
be the chief chain guaranteed by rank supersolvability, where $\rho(m_{\lambda})=\lambda$.
Let $A$ be an antichain cutset.

We will prove the following generalization of Theorem~\ref{thm:Main}.
\begin{thm}
\label{thm:MainExt}Let $R\subseteq[-\infty,\infty]$ be an interval.
Let $L$ be a rank supersolvable lattice with respect to the $R$-grading
$\rho$, with chief chain $\mathbf{m}$. Suppose further that the
following conditions are met for every maximal chain $\mathbf{c}$,
$\lambda_{0}\in R$, and $z\in L$:
\begin{alignat*}{2}
\sup\big\{ m_{\lambda_{0}}\wedge c & :c\in\mathbf{c},\rho(c)<\infty\big\}=m_{\lambda_{0}}\quad & \quad\inf\big\{ m_{\lambda_{0}}\vee c & :c\in\mathbf{c},\rho(c)>-\infty\big\}=m_{\lambda_{0}}\\
\sup\big\{ m_{\lambda}\wedge z & :\lambda\in R,\lambda<\infty\big\}=z & \inf\big\{ m_{\lambda}\vee z & :\lambda\in R,\lambda>-\infty\big\}=z.
\end{alignat*}
If $A$ is an antichain cutset for $L$, then there is some grading
$\sigma$ under which $A$ is a level set.
\end{thm}

We notice that if $R$ is bounded from above (in $\rr$), then the
supremum conditions are automatic; similarly if $R$ is bounded from
below with the infimum conditions. Indeed, in light of Lemma~\ref{lem:CtsAtInfty}
and the discussion surrounding, we see that the additional conditions
give exactly that the maps of Corollaries~\ref{cor:JoinMeetCtntyC}
and \ref{cor:JoinMeetCtntyM} are continuous also at $\infty$ and/or
at $-\infty$ (if present in $R$). We summarize with a lemma:
\begin{lem}
\label{lem:CtsJoinMeet}Under the conditions of Theorem~\ref{thm:MainExt},
writing $\mathbf{c}=\left\{ c_{\kappa}:\kappa\in R\right\} $, the
following maps $R\to R$ are continuous.
\begin{align*}
\lambda\mapsto\rho(m_{\lambda}\wedge z)\qquad & \lambda\mapsto\rho(m_{\lambda}\vee z)\\
\kappa\mapsto\rho(m_{\lambda_{0}}\wedge c_{\kappa})\qquad & \kappa\mapsto\rho(m_{\lambda_{0}}\vee c_{\kappa})
\end{align*}
\end{lem}

\begin{rem}
Examination of Examples~\ref{exa:DisCtsAtInfty} and \ref{exa:BddMeasBool}
demonstrates that continuity at $\pm\infty$ is not automatic. Since
the lattice of Example~\ref{exa:DisCtsAtInfty} is isomorphic to
a space-time causality order, the conclusion of Theorem~\ref{thm:MainExt}
nonetheless holds by \cite{Foldes:2024}. We do not know whether the
antichain cutsets of the lattice in Example~\ref{exa:BddMeasBool}
are level sets in some $[0,\infty)$-grading.
\end{rem}

We now construct the new rank function $\sigma$ for Theorem~\ref{thm:MainExt}.
Consider the \emph{good chain} through $z$ given by $\left\{ z\wedge m_{\lambda}\,:\,\lambda\in R\right\} \cup\left\{ z\vee m_{\lambda}\,:\,\lambda\in R\right\} $.
We will show that the good chain through $z$ is a maximal chain,
hence it intersects with $A$ at a unique point $\alpha(z)$. We define
\[
\sigma(z):=\rho(z)-\rho(\alpha(z)).
\]
It is obvious by construction that $\sigma(a)=0$ if and only if $a\in A$,
so that $A$ is a level set with respect to $\sigma$. We will show
$\sigma$ to be a new $[\sigma(\hat{0}),\sigma(\hat{1})]$-rank function,
which can then be rescaled to an $R$-grading. We will need to check
that $\sigma$ is well-defined, strictly-increasing, and surjective
onto $[\sigma(\hat{0}),\sigma(\hat{1})]$ over any maximal chain.

First, in order for $\sigma$ to be well-defined, we need to check
that the good chain is maximal.
\begin{lem}
\label{lem:GoodChain}The good chain through $z$ is a maximal chain.
\end{lem}

\begin{proof}
It suffices to show that elements of the good chain achieve every
rank in $R$.

It follows from Lemma~\ref{lem:CtsJoinMeet} that $\rho(\left\{ z\wedge m_{\lambda}\,:\,\lambda\in R\right\} )$
is the continuous image of a connected set having upper bound $\rho(z)$;
similarly for $\rho(\left\{ z\vee m_{\lambda}\,:\,\lambda\in R\right\} )$.
As these share the point $\rho(z)$, the ranks of a good chain form
a connected subset of $R$. Since $\hat{0}=\hat{0}\wedge z$ and $\hat{1}=\hat{1}\vee z$,
the desired follows.
\end{proof}
An entirely similar proof gives:
\begin{lem}
\label{lem:ReversedGoodChain}If $\mathbf{c}$ is a maximal chain
and $m$ is on the given chief chain $\mathbf{m}$ of $L$ , then
the chain $\{m\wedge c:c\in\mathbf{c}\}\cup\{m\vee c:c\in\mathbf{c}\}$
is a maximal chain.
\end{lem}

\begin{rem}
By Corollary~\ref{cor:JoinMeetCtntyC}, if $L$ is $R$-graded for
$R$ a bounded subset of $\rr$, then we may relax the requirement
that $m$ lie on the given chief chain to simple rank modularity.
\end{rem}

We will mostly focus on elements $z$ above the antichain cutset $A$;
elements below the antichain cutset may be dealt with in a completely
symmetrical manner.

We start with an easy observation:
\begin{lem}
\label{lem:OrderLambdasFromAlphas}If $w<z$ are above or on the antichain
cutset, so that $\alpha(w)=m_{\lambda}\wedge w$ and $\alpha(z)=m_{\kappa}\wedge z$,
then either $\lambda>\kappa$ or else $\alpha(w)=\alpha(z)$.
\end{lem}

\begin{proof}
If $\lambda\leq\kappa$, then $m_{\lambda}\wedge w\leq m_{\kappa}\wedge z$.
Since both lie on the antichain $A$, equality follows.
\end{proof}
See the illustration in Figure~\ref{fig:ProjectionA}. We can now
obtain the increasing property:
\begin{lem}
\label{lem:SigmaIncreasing}The map $\sigma$ is strictly increasing.
\end{lem}

\begin{proof}
We need to show that if $w<z$, then $\sigma(w)<\sigma(z)$.

We first check the case where $w$ and $z$ are both above or on $A$.
We write $\alpha(w)=m_{\lambda}\wedge w$ and $\alpha(z)=m_{\kappa}\wedge z$.
If $\alpha(w)=\alpha(z)$, then the result is immediate since $\rho$
is strictly increasing. Thus, we may assume by Lemma~\ref{lem:OrderLambdasFromAlphas}
that $\lambda>\kappa$. Now 
\begin{align*}
\sigma(z)-\sigma(w) & =\left(\rho(z)-\rho(w)\right)-\left(\rho(m_{\kappa}\wedge z)-\rho(m_{\lambda}\wedge w)\right)\\
 & >\left(\rho(z)-\rho(w)\right)-\left(\rho(m_{\kappa}\wedge z)-\rho(m_{\kappa}\wedge w)\right)\\
 & \geq0.
\end{align*}
where the strict inequality is because $m_{\kappa}\wedge w=\alpha(w)\wedge\alpha(z)<\alpha(w)$,
and the weak inequality is by Corollary~\ref{cor:RMdiamondbounds}.

Finally, if both are below or on $A$, then the result follows by
duality. If $w$ is below $A$ and $z$ is above, then there is by
definition an element $x$ of $A$ on a maximal chain between them,
and then $\sigma(w)<\sigma(x)<\sigma(z)$.
\end{proof}
In order to show that $\sigma$ is bijective over $\mathbf{c}$ onto
$[\sigma(\hat{0}),\sigma(\hat{1})]$, it now suffices to show that
its image is connected. This will be an easy consequence of the following
lemma; compare with the discussion following Corollaries~\ref{cor:JoinMeetCtntyC}
and \ref{cor:JoinMeetCtntyM}.

\begin{figure}
\medskip{}
\includegraphics{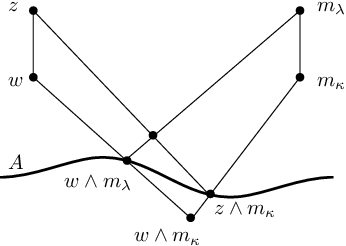}

\caption{\label{fig:ProjectionA}Projection to $A$.}
\end{figure}

\begin{lem}
\label{lem:RankDiffLipschitz}If $\mathbf{c}$ is a maximal chain,
then $\hat{\alpha}=\rho\circ\alpha\circ\rho\vert_{\mathbf{c}}^{-1}$
is a continuous map $R\to\mathbb{R}$.

That is, the rank of $\alpha(z)$ varies continuously over the chain
$\mathbf{c}$.
\end{lem}

\begin{proof}
We prove that $\hat{\alpha}$ is continuous from the right at a point
$w$ that is above or on the antichain cutset $A$. Continuity from
the left above $A$ is entirely similar, and points below or on the
antichain cutset follow from duality.

By Lemma~\ref{lem:CtsJoinMeet} and the fact that the continuous
preimage of a point is closed, there is a minimal $\lambda$ with
$m_{\lambda}\wedge w=\alpha(w)$. It follows that $m_{\lambda}\geq\alpha(w)$.
If $m_{\lambda}=\alpha(w)$, then $m_{\lambda}=\alpha(z)$ also for
every $z>w$, and a constant function is right-continuous.

Thus, we may assume that $m_{\lambda}>\alpha(w)$. By Lemmas~\ref{lem:CtsJoinMeet}
and \ref{lem:GoodChain} and minimality of $\lambda$, we can choose
an $m_{\kappa}$ that is above the antichain cutset, below $m_{\lambda}$,
and so that $\rho(m_{\lambda}\wedge w)-\rho(m_{\kappa}\wedge w)<\epsilon/2$.
By Lemma~\ref{lem:ReversedGoodChain}, the chain $\left\{ m_{\kappa}\wedge c:c\in\mathbf{c}\right\} $
is maximal on the interval $[\hat{0},m_{\kappa}]$, hence there is
some $z_{0}>w$ so that $m_{\kappa}\wedge z_{0}=\alpha(z_{0})$. Now,
since $m_{\kappa}\wedge z\leq\alpha(z_{0})$ and $m_{\lambda}\wedge z\geq\alpha(w)$,
every $z$ between $w$ and $z_{0}$ has $\alpha(z)$ between $m_{\kappa}\wedge z$
and $m_{\lambda}\wedge z$. In particular, if $z\in\mathbf{c}$ is
chosen so $\rho(m_{\lambda}\wedge z)<\min(\rho(m_{\lambda}\wedge w)+\epsilon/2,\rho(m_{\lambda}\wedge z_{0}))$,
then $\alpha(z)$ and $\alpha(w)$ are both on an interval $[m_{\kappa}\wedge w,m_{\lambda}\wedge z]$.
Since $\rho(m_{\lambda}\wedge z)-\rho(m_{\kappa}\wedge w)<\epsilon$,
the result now follows.
\end{proof}
\begin{cor}
\label{cor:SigmaSurjective}For any maximal chain $\mathbf{c}$, it
holds that $\sigma\circ\rho\vert_{\mathbf{c}}^{-1}$ is a continuous
map, hence that $\sigma\vert_{\mathbf{c}}$ is surjective onto the
interval $[\sigma(\hat{0}),\sigma(\hat{1})]$.
\end{cor}

\begin{proof}
We have that $\sigma\circ\rho\vert_{\mathbf{c}}^{-1}=\mathrm{id}_{R}-\hat{\alpha}$,
yielding continuity. Now the image of $\sigma\vert_{\mathbf{c}}$
is a connected subset of $[\sigma(\hat{0}),\sigma(\hat{1})]$ that
contains both endpoints.
\end{proof}
Theorem~\ref{thm:Main} now follows from Lemma~\ref{lem:SigmaIncreasing}
and Corollary~\ref{cor:SigmaSurjective}.
\begin{example}
The chief chain may not be rank modular with respect to the new rank
function $\sigma$. Consider the measurable Boolean lattice $\mathcal{B}$
on $[0,2]$ with Lesbegue measure $\mu$. Let $\nu(U)$ be the integral
over $U$ of the function that is $1$ on $[0,1]$ and $2$ on $[1,2]$,
and let $A$ be the family of all sets of measure $1$ with respect
to $\nu$, an antichain cutset in $\mathcal{B}$. Since $\mathcal{B}$
is rank modular, any maximal chain is a chief chain. We take the chief
chain consisting of the intervals $[0,t]$ for $0\leq t\leq2$.

But now $\sigma([0,t])=\mu([0,t])-\mu([0,1])=t-1$, while $\sigma([1,2])=\mu([1,2])-\mu([1,3/2])=1/2$.
Now $\sigma([0,1])+\sigma([1,2])\neq\sigma([0,2])+\sigma(\emptyset)$.
Thus, the element $[0,1]$ is not rank modular with respect to $\sigma$!
\end{example}

\bibliographystyle{10_Users_russw_Documents_Research_mypapers_Antichain_cutsets_in_real-ranked_lattices_hamsplain}
\bibliography{9_Users_russw_Documents_Research_Master}

\def\cprime{$'$}
\providecommand{\bysame}{\leavevmode\hbox to3em{\hrulefill}\thinspace}
\providecommand{\href}[2]{#2}
\begin{thebibliography}{10}

\bibitem{Birkhoff:1967}
Garrett Birkhoff, \emph{Lattice theory}, 3rd ed., American Mathematical Society
  Colloquium Publications, vol.~25, American Mathematical Society, Providence,
  R.I., 1979.

\bibitem{Bjorner:1987}
Anders Bj{\"o}rner, \emph{Continuous partition lattice}, Proc. Nat. Acad. Sci.
  U.S.A. \textbf{84} (1987), no.~18, 6327--6329.

\bibitem{Bjorner:2019}
Anders Bj\"{o}rner, \emph{Continuous matroids revisited}, Building bridges
  {II}---mathematics of {L}\'{a}szl\'{o} {L}ov\'{a}sz, Bolyai Soc. Math. Stud.,
  vol.~28, Springer, Berlin, 2019, pp.~17--28.

\bibitem{Bjorner/Lovasz:1987}
Anders Bj\"{o}rner and L\'{a}szl\'{o} Lov\'{a}sz, \emph{Pseudomodular lattices
  and continuous matroids}, Acta Sci. Math. (Szeged) \textbf{51} (1987),
  no.~3-4, 295--308.

\bibitem{Contreras-Mantilla/Sinclair:2025UNP}
Jos{\'e} Contreras~Mantilla and Thomas Sinclair, \emph{The model theory of
  metric lattices: pseudofinite partition lattices}, preprint, 2025,
  {arXiv:2507.10932}.

\bibitem{Engel/Mitsis/Pelekis/Reiher:2020}
Konrad Engel, Themis Mitsis, Christos Pelekis, and Christian Reiher,
  \emph{Projection inequalities for antichains}, Israel J. Math. \textbf{238}
  (2020), no.~1, 61--90.

\bibitem{Foldes:2024}
Stephan Foldes, \emph{Remarks on antichains in the causality order of
  space-time}, J. Mult.-Valued Logic Soft Comput. \textbf{42} (2024), no.~4,
  375--383.

\bibitem{Foldes/Woodroofe:2013}
Stephan Foldes and Russ Woodroofe, \emph{Antichain cutsets of strongly
  connected posets}, Order \textbf{30} (2013), no.~2, 351--361,
  {arXiv:1109.5705}.

\bibitem{Foldes/Woodroofe:2021}
\bysame, \emph{A modular characterization of supersolvable lattices}, Proc.
  Amer. Math. Soc. \textbf{150} (2021), no.~1, 31--39, {arXiv:2011.11657}.

\bibitem{Gierz/Hofmann/Keimel/Lawson/Mislove/Scott:2003}
G.~Gierz, K.~H. Hofmann, K.~Keimel, J.~D. Lawson, M.~Mislove, and D.~S. Scott,
  \emph{Continuous lattices and domains}, Encyclopedia of Mathematics and its
  Applications, vol.~93, Cambridge University Press, Cambridge, 2003.

\bibitem{Ginsburg/Sands:2025UNP}
John Ginsburg and Bill Sands, \emph{Cutsets in $\mathcal{P}({X})$}, preprint,
  {arXiv:2508.10221}.

\bibitem{Graham:1978}
R.~L. Graham, \emph{Maximum antichains in the partition lattice}, Math.
  Intelligencer \textbf{1} (1978), no.~2, 84--86.

\bibitem{Haiman:1994}
Mark~D. Haiman, \emph{On realization of {B}j\"orner's ``continuous partition
  lattice'' by measurable partitions}, Trans. Amer. Math. Soc. \textbf{343}
  (1994), no.~2, 695--711.

\bibitem{McNamara:2003}
Peter McNamara, \emph{E{L}-labelings, supersolvability and 0-{H}ecke algebra
  actions on posets}, J. Combin. Theory Ser. A \textbf{101} (2003), no.~1,
  69--89, {arXiv:math/0111156}.

\bibitem{Monjardet:1981}
Bernard Monjardet, \emph{Metrics on partially ordered sets---a survey},
  Discrete Math. \textbf{35} (1981), 173--184.

\bibitem{Rival/Zaguia:1985}
Ivan Rival and Nejib Zaguia, \emph{Antichain cutsets}, Order \textbf{1} (1985),
  no.~3, 235--247.

\bibitem{Stanley:1972}
Richard~P. Stanley, \emph{Supersolvable lattices}, Algebra Universalis
  \textbf{2} (1972), 197--217.

\bibitem{Stern:1999}
Manfred Stern, \emph{Semimodular lattices. {T}heory and applications},
  Encyclopedia of Mathematics and its Applications, vol.~73, Cambridge
  University Press, Cambridge, 1999.

\bibitem{Vershik/Yakubovich:2001}
Anatoly~M. Vershik and Yuri~V. Yakubovich, \emph{Continuous lattices of
  partitions and lattices of continuous partitions}, Proceedings of the {S}t.
  {P}etersburg {M}athematical {S}ociety, {V}ol. {VII}, Amer. Math. Soc. Transl.
  Ser. 2, vol. 203, Amer. Math. Soc., Providence, RI, 2001, pp.~1--15.

\bibitem{VonNeumann:1936a}
J.~{von Neumann}, \emph{{Continuous geometry.}}, {Proc. Natl. Acad. Sci. USA}
  \textbf{22} (1936), 92--100 (English).

\bibitem{VonNeumann:1936b}
\bysame, \emph{{Examples of continuous geometries.}}, {Proc. Natl. Acad. Sci.
  USA} \textbf{22} (1936), 101--108 (English).

\end{thebibliography}

\end{document}